\documentclass[11pt]{amsart}
\usepackage{amsmath,amssymb,mathtools}
\usepackage{hyperref}
\usepackage[margin=1.1in]{geometry}

\newtheorem{theorem}{Theorem}[section]
\newtheorem{proposition}[theorem]{Proposition}
\newtheorem{lemma}[theorem]{Lemma}
\newtheorem{corollary}[theorem]{Corollary}
\theoremstyle{remark}
\newtheorem{remark}[theorem]{Remark}

\newcommand{\C}{\mathbf C}
\newcommand{\Hh}{\mathfrak H}
\newcommand{\SLtwoZ}{\mathrm{SL}_2(\mathbf Z)}
\newcommand{\Tr}{\operatorname{Tr}}
\newcommand{\Span}{\operatorname{Span}}
\newcommand{\RePart}{\operatorname{Re}}
\newcommand{\one}{\mathbf 1}

\title[Simultaneous non-vanishing]{Simultaneous nonvanishing of quadratic twists
via Rankin--Cohen brackets}

\author{Ramin Takloo-Bighash}

\subjclass[2020]{Primary 11F67; Secondary 11F11}
\keywords{Rankin--Cohen brackets, Eisenstein series, central values,
quadratic twists, linear independence}

\begin{document}

\begin{abstract}
Let $D$ be an odd fundamental discriminant, with $D=1$ permitted, and let
$r\geq 1$ be fixed.  We prove that, for every sufficiently large integer
$\ell$ satisfying $(-1)^\ell D>0$, the first $r$ traced diagonal
Rankin--Cohen brackets
\[
 \Tr_1^{|D|}[G_{\ell-2e,D},G_{\ell-2e,D}]_{2e},
 \qquad 1\leq e\leq r,
\]
are linearly independent in $S_{2\ell}(\SLtwoZ)$.  Here $G_{k,D}$ is the
Eisenstein series of weight $k$, level $|D|$, and nebentypus $\chi_D$.
The Petersson formula of Kayath, Lane, Neifeld, Ni, and Xue then implies
that at least $r$ normalized Hecke eigenforms
$f\in S_{2\ell}(\SLtwoZ)$ satisfy
$L(f\otimes\chi_D,\ell)\neq 0$.  For $D=1$, this gives, for every fixed
$r$ and every sufficiently large $K\equiv 0\pmod 4$, at least $r$
level-one eigenforms of weight $K$ with nonzero central value. 
\end{abstract}

\maketitle

\section{Introduction}

Rankin--Cohen brackets provide a natural way to construct cusp forms from
modular forms of lower weight.  They originate in work of Rankin and Cohen
and were placed in a broad structural framework by Zagier; see
\cite{Rankin,Cohen,ZagierDifferential}. 

Lanphier and Takloo-Bighash \cite[Cor. 3.2]{LanphierTaklooBighash} used the fact that $[E_{K/2-2},E_{K/2-2}]_2$ is not an eigenform for  $K>20$ with $K\equiv 0\pmod 4$ to prove the existence of at least two eigenforms
$f\in S_K(\SLtwoZ)$ with $L(f,K/2)\neq 0$.
 The purpose of this note is to generalize this result. 

In this work we will use a construction by Kayath, Lane, Neifeld, Ni, and Xue
\cite{KayathEtAl}.  For an odd fundamental discriminant $D$ and integers
$k\geq 4$, $e>0$, with $\ell=k+2e$ and $(-1)^\ell D>0$, they introduced
the level-one cusp form
\begin{equation}\label{eq:intro-kernel}
 \mathcal F_{D,k,e}:=\Tr_1^{|D|}[G_{k,D},G_{k,D}]_{2e}
 \in S_{2\ell}(\SLtwoZ).
\end{equation}
See \S \ref{sec:preliminaries} for the notation used here and also the result we will need from \cite{KayathEtAl}.

As usual, $\chi_D(n)=\left(\frac Dn\right)$ denotes the primitive quadratic
character attached to $D$.  We follow \cite{KayathEtAl} in including
$D=1$ among the fundamental discriminants, with $\chi_1$ the trivial
character.

\begin{theorem}\label{thm:main}
Fix an odd fundamental discriminant $D$ and an integer $r\geq 1$.  There is
an effective integer $\ell_0=\ell_0(D,r)$ such that, whenever
\[
 \ell\geq \ell_0 \qquad\text{and}\qquad (-1)^\ell D>0,
\]
the $r$ cusp forms
\begin{equation}\label{eq:initial-family}
 \mathcal F_{\ell,e}^{(D)}:=\mathcal F_{D,\ell-2e,e}
 =\Tr_1^{|D|}[G_{\ell-2e,D},G_{\ell-2e,D}]_{2e},
 \qquad 1\leq e\leq r,
\end{equation}
are linearly independent in $S_{2\ell}(\SLtwoZ)$.

Let $\mathcal B_{2\ell}$ denote the normalized Hecke eigenbasis of
$S_{2\ell}(\SLtwoZ)$.  Then
\begin{equation}\label{eq:nonvanishing-count}
 \#\{f\in\mathcal B_{2\ell}:L(f\otimes\chi_D,\ell)\neq 0\}\geq r.
\end{equation}
\end{theorem}

For $D=1$, the trace is trivial and $G_{k,1}$ is a nonzero scalar multiple
of the normalized level-one Eisenstein series $E_k$.

\begin{corollary}\label{cor:untwisted}
For every fixed integer $r\geq 1$, there exists $K_0(r)$ such that, for
every $K\geq K_0(r)$ with $K\equiv 0\pmod 4$, the forms
\[
 [E_{K/2-2e},E_{K/2-2e}]_{2e},\qquad 1\leq e\leq r,
\]
are linearly independent in $S_K(\SLtwoZ)$.  In particular,
\[
 \#\{f\in\mathcal B_K:L(f,K/2)\neq 0\}\geq r,
\]
where $\mathcal B_K$ denotes the normalized Hecke eigenbasis of
$S_K(\SLtwoZ)$.
\end{corollary}

The theorem keeps $D$ and $r$ fixed while $\ell$ tends to infinity.  We
give an effective method to find $\ell$, but do not optimize its dependence on $r$ or
assert a positive proportion of nonvanishing central values.  

\

It is clear that this work owes a great deal of intellectual and technical debt to \cite{KayathEtAl}. In developing this work we benefited from running some experiments using Chatgpt 6 Astra provided via Openai's ChatGPT for Academic Researchers granted to Julius Ross who kindly provided us access. The author is partially supported by a Collaboration Grant from the Simons Foundation.

\section{Eisenstein series, traces, and Petersson kernels}
\label{sec:preliminaries}

We recall the formulas from \cite{KayathEtAl} that are needed in the proof of the main theorem.
Throughout, $q=e^{2\pi iz}$. For a modular form $F$ we set 
\[
 F^{(j)}(z):=\left(\frac{1}{2\pi i}\frac{d}{dz}\right)^jF(z)
 =\left(q\frac{d}{dq}\right)^jF(z).
\]
For modular forms $F$ and $H$ of respective weights $a$ and $b$, their
$m$th Rankin--Cohen bracket is
\begin{equation}\label{eq:RC-definition}
 [F,H]_m=\sum_{j=0}^m(-1)^j
 \binom{m+a-1}{m-j}\binom{m+b-1}{j}F^{(j)}H^{(m-j)}.
\end{equation}
See \cite{Cohen,ZagierDifferential}.  The brackets used below have
order $2e\geq 2$ and are cusp forms. 

If $w$ is an even integer and $H$ is a weight-$w$
modular form of level $N$ and trivial character, we define 
\begin{equation}\label{eq:trace-definition}
 \Tr_1^N H:=\sum_{\gamma\in\Gamma_0(N)\backslash\SLtwoZ}H|_w\gamma,
 \qquad
 (H|_w\gamma)(z):=(cz+d)^{-w}H(\gamma z),
\end{equation}
where $\gamma=\left(\begin{smallmatrix}a&b\\c&d\end{smallmatrix}\right)$.
This is the unnormalized trace used in \cite{KayathEtAl}.  The Petersson
inner product on $S_w(\SLtwoZ)$ is
\begin{equation}\label{eq:Petersson-definition}
 \langle F,H\rangle:=
 \int_{\SLtwoZ\backslash\Hh}F(z)\overline{H(z)}y^w\,\frac{dx\,dy}{y^2}.
\end{equation}

For a fundamental discriminant $A$ and $\RePart(s)>1$, write
\[
 L_A(s):=L(s,\chi_A)=\sum_{n\geq 1}\frac{\chi_A(n)}{n^s}.
\]
We use the same notation for its meromorphic continuation.
Let $k\geq 4$, let $D$ be an odd fundamental discriminant satisfying
$(-1)^kD>0$, and suppose that $D=D_1D_2$ is an ordered factorization into
relatively prime fundamental discriminants.  Following
\cite[Section~4]{KayathEtAl}, put
\begin{equation}\label{eq:Eis-general}
 G_{k,D_1,D_2}(z)=\sum_{n\geq 0}\sigma_{k-1,D_1,D_2}(n)q^n,
\end{equation}
where
\begin{equation}\label{eq:twisted-divisor-sum}
 \sigma_{k-1,D_1,D_2}(n)=
 \begin{cases}
 -L_{D_1}(1-k)L_{D_2}(0),&n=0,\\[3pt]
 \displaystyle\sum_{ab=n}\chi_{D_1}(a)\chi_{D_2}(b)a^{k-1},&n>0.
 \end{cases}
\end{equation}
With these parity and coprimality hypotheses,
\[
 G_{k,D_1,D_2}\in M_k(\Gamma_0(|D|),\chi_D).
\]
The constant term in \eqref{eq:twisted-divisor-sum} vanishes unless
$D_2=1$.  Indeed, if
$D_2\neq 1$ is positive, then $\chi_{D_2}$ is a nontrivial even character
and $L_{D_2}(0)=0$.  If $D_2<0$, then $\chi_{D_2}(-1)=-1$; since
$\chi_D(-1)=(-1)^k$, we have
$\chi_{D_1}(-1)=(-1)^{k+1}$, and the parity trivial zero gives
$L_{D_1}(1-k)=0$.  In the exceptional case $D_2=1$,
\begin{equation}\label{eq:constant-term}
 \sigma_{k-1,D,1}(0)=\frac{L_D(1-k)}2,
\end{equation}
because $L_1(0)=\zeta(0)=-1/2$.  We abbreviate
\begin{equation}\label{eq:GkD}
 G_{k,D}:=G_{k,D,1}
 =\frac{L_D(1-k)}2
 +\sum_{n\geq 1}\left(\sum_{a\mid n}\chi_D(a)a^{k-1}\right)q^n.
\end{equation}
Then $G_{k,D}\in M_k(\Gamma_0(|D|),\chi_D)$.  Since $G_{k,D}$ has nebentypus $\chi_D$,
the bracket $[G_{k,D},G_{k,D}]_{2e}$ has nebentypus $\chi_D^2$, which is
trivial on $\Gamma_0(|D|)$.  Thus the trace in
\eqref{eq:intro-kernel} is a trace of a modular form with trivial
character.

For a $q$-series $H(z)=\sum_{n\geq 0}c_H(n)q^n$, let
\[
 U_mH(z):=\sum_{n\geq 0}c_H(mn)q^n.
\]
The following trace decomposition is \cite[Proposition~6.3]{KayathEtAl}.

\begin{proposition}[Kayath--Lane--Neifeld--Ni--Xue]
\label{prop:trace-decomposition}
Let $k\geq 4$, $e>0$, $\ell=k+2e$, and let $D$ be an odd fundamental
discriminant satisfying $(-1)^\ell D>0$.  Then
\begin{equation}\label{eq:trace-decomposition}
 \mathcal F_{D,k,e}=
 \sum_{D=D_1D_2}\chi_{D_2}(-1)|D_2|^{-2e}
 U_{|D_2|}[G_{k,D_1,D_2},G_{k,D_1,D_2}]_{2e},
\end{equation}
where the sum runs over ordered factorizations of $D$ into relatively prime
fundamental discriminants.
\end{proposition}

The second input is the Petersson inner product formula
\cite[Proposition~5.6]{KayathEtAl}.  For a normalized Hecke eigenform
$f(z)=\sum_{n\geq 1}a_f(n)q^n\in S_{2\ell}(\SLtwoZ)$, set
\[
 L(f,s)=\sum_{n\geq 1}\frac{a_f(n)}{n^s},\qquad
 L(f\otimes\chi_D,s)=\sum_{n\geq 1}\frac{\chi_D(n)a_f(n)}{n^s}.
\]
These functions analytically continue to the whole complex plane and we use the same notation for the analytic continuation. 

\begin{proposition}[Kayath--Lane--Neifeld--Ni--Xue]
\label{prop:Petersson-formula}
Under the hypotheses of Proposition~\ref{prop:trace-decomposition}, one has
\begin{align}\label{eq:Petersson-formula}
 \langle \mathcal F_{D,k,e},f\rangle
 &=\frac12
 \frac{\Gamma(2k+4e-1)\Gamma(k+2e)}
 {(2e)!(4\pi)^{2k+4e-1}\Gamma(k)}
 \frac{L_D(1-k)}{L_D(k)} \notag\\
 &\hspace{35mm}\times
 L(f,2k+2e-1)L(f\otimes\chi_D,k+2e).
\end{align}
\end{proposition}

\section{An asymptotic formula for Fourier coefficients}
\label{sec:endpoint}

Write
\begin{equation}\label{eq:coeff-def}
 \mathcal F_{D,k,e}(z)=\sum_{n\geq 1}A_{D,k,e}(n)q^n
\end{equation}
and set
\begin{equation}\label{eq:B-def}
 B_{k,e}:=\binom{k+2e-1}{2e}.
\end{equation}
For a fixed Fourier index, the terms in which one Eisenstein coefficient is
constant dominate all terms in which both indices are positive.

For a fundamental discriminant $D$, we let $\nu(D)$ be the number of ordered factorizations
$D=D_1D_2$ into relatively prime fundamental discriminants.

\begin{proposition}\label{prop:endpoint}
Fix an odd fundamental discriminant $D$, an integer $e\geq 1$, and a
positive integer $n$.  Let $k\geq 4$ tend to infinity through integers
satisfying $(-1)^kD>0$.  Then
\begin{equation}\label{eq:endpoint-asymptotic}
 \frac{A_{D,k,e}(n)}{B_{k,e}L_D(1-k)}
 -n^{2e}\sigma_{k-1,D,1}(n)\longrightarrow 0.
\end{equation}
More precisely, with $M=n|D|$,
\begin{align}\label{eq:effective-error}
 \left|\frac{A_{D,k,e}(n)}{B_{k,e}L_D(1-k)}
 -n^{2e}\sigma_{k-1,D,1}(n)\right|
 \leq \frac{\zeta(2)}2\nu(D)C_eM^{2e+1}|D|^{1/2}
 \frac{(2\pi M^2/|D|)^k}{\Gamma(k)}.
\end{align}
In particular, the convergence in \eqref{eq:endpoint-asymptotic} is
effective.
\end{proposition}

We first record an elementary inequality that will be used in the proof of the proposition. 

\begin{lemma}\label{lem:binomial}
Let $k\geq 4$, $e\geq 1$, and $0\leq j\leq 2e$.  Then
\begin{equation}\label{eq:binomial-bound}
 \frac{\binom{k+2e-1}{j}\binom{k+2e-1}{2e-j}}
 {\binom{k+2e-1}{2e}}
 \leq \binom{2e}{j}\left(\frac{2e+3}{4}\right)^{2e}.
\end{equation}
Consequently, for $a,b\geq 0$,
\begin{equation}\label{eq:C-bound}
 \left|\sum_{j=0}^{2e}(-1)^ja^jb^{2e-j}
 \binom{k+2e-1}{2e-j}\binom{k+2e-1}{j}\right|
 \leq B_{k,e}C_e(a+b)^{2e},
\end{equation}
where
\begin{equation}\label{eq:C-def}
 C_e:=\left(\frac{2e+3}{2}\right)^{2e}.
\end{equation}
\end{lemma}

\begin{proof}
Put $N=k+2e-1$.  Since $N-2e+1=k$, we have
\[
 \binom N{2e}
 =\frac{N(N-1)\cdots(N-2e+1)}{(2e)!}
 \geq \frac{k^{2e}}{(2e)!}.
\]
On the other hand,
\[
 \binom Nj\binom N{2e-j}
 \leq \frac{N^{2e}}{j!(2e-j)!}.
\]
The ratio of these two estimates is
\[
 \binom{2e}{j}\left(\frac Nk\right)^{2e}
 =\binom{2e}{j}\left(1+\frac{2e-1}{k}\right)^{2e},
\]
which is at most the right-hand side of
\eqref{eq:binomial-bound}, since $k\geq 4$.  Finally,
$a^jb^{2e-j}\leq(a+b)^{2e}$, and summing
\eqref{eq:binomial-bound} over $j$ gives
\[
 \left(\frac{2e+3}{4}\right)^{2e}
 \sum_{j=0}^{2e}\binom{2e}{j}
 =\left(\frac{2e+3}{2}\right)^{2e}.
\]
This proves \eqref{eq:C-bound}.
\end{proof}

\begin{proof}[Proof of Proposition \ref{prop:endpoint}]
By Proposition~\ref{prop:trace-decomposition}, the coefficient
$A_{D,k,e}(n)$ is a finite sum indexed by the ordered factorizations
$D=D_1D_2$.  If
\[
 G_{k,D_1,D_2}(z)=\sum_{m\geq 0}\sigma_{k-1,D_1,D_2}(m)q^m,
\]
then the $q^n$-coefficient contributed by $(D_1,D_2)$ equals
\begin{equation}\label{eq:coefficient-decomposition}
 \chi_{D_2}(-1)|D_2|^{-2e}
 \sum_{\substack{a,b\geq 0\\a+b=n|D_2|}}
 \sigma_{k-1,D_1,D_2}(a)\sigma_{k-1,D_1,D_2}(b)C_{e,k}(a,b),
\end{equation}
where
\begin{equation}\label{eq:Cek}
 C_{e,k}(a,b):=\sum_{j=0}^{2e}(-1)^ja^jb^{2e-j}
 \binom{k+2e-1}{2e-j}\binom{k+2e-1}{j}.
\end{equation}
This expression is the first form of the coefficient calculation in \cite[Lemma~6.5]{KayathEtAl}.

We separate \eqref{eq:coefficient-decomposition} into its endpoint terms,
where $ab=0$, and its interior remainder, where $a,b\geq 1$.  The constant
term in \eqref{eq:Eis-general} vanishes unless $(D_1,D_2)=(D,1)$, by the
trivial-zero argument following \eqref{eq:twisted-divisor-sum}.  For this
factorization, the endpoint $(a,b)=(0,n)$ contributes
\[
 \frac{L_D(1-k)}2\sigma_{k-1,D,1}(n)n^{2e}B_{k,e},
\]
because only $j=0$ in \eqref{eq:Cek} survives.  The endpoint $(a,b)=(n,0)$
gives the same contribution: only $j=2e$ survives, and $(-1)^{2e}=1$.
Thus the two endpoint terms sum to
\begin{equation}\label{eq:endpoint-exact}
 B_{k,e}L_D(1-k)n^{2e}\sigma_{k-1,D,1}(n).
\end{equation}

It remains to bound the interior contribution $R_{D,k,e}(n)$.  For every
positive integer $m$,
\begin{equation}\label{eq:sigma-bound}
 |\sigma_{k-1,D_1,D_2}(m)|
 \leq\sum_{uv=m}u^{k-1}\leq\tau(m)m^{k-1}\leq m^k.
\end{equation}
Here $\tau(m)$ is the divisor-counting function.  For every term occurring
in \eqref{eq:coefficient-decomposition},
\[
 a+b=n|D_2|\leq n|D|=M.
\]
Lemma~\ref{lem:binomial} and \eqref{eq:sigma-bound} therefore give
\[
 |\sigma_{k-1,D_1,D_2}(a)\sigma_{k-1,D_1,D_2}(b)C_{e,k}(a,b)|
 \leq B_{k,e}C_eM^{2e}M^{2k}.
\]
There are at most $M$ positive pairs $(a,b)$ for each factorization, and
$|D_2|^{-2e}\leq 1$.  Hence
\begin{equation}\label{eq:remainder-bound}
 |R_{D,k,e}(n)|\leq\nu(D)C_eM^{2e+1}B_{k,e}M^{2k}.
\end{equation}

Under the admissible parity condition, the functional equation of the
primitive quadratic Dirichlet $L$-function gives
\begin{equation}\label{eq:Dirichlet-FE-value}
 |L_D(1-k)|=\frac{2\Gamma(k)}{(2\pi)^k}|D|^{k-1/2}L_D(k);
\end{equation}
see \cite[Chapter~9]{Davenport}.
For $D\neq1$, the primitive Gauss sum has absolute value
$\sqrt{|D|}$, and the phase factor in that functional equation has
absolute value $2$ because $\chi_D(-1)=(-1)^k$.  For $D=1$, the same
identity follows from the zeta functional equation.  Since $k>1$, the Euler product is positive and
\begin{equation}\label{eq:Dirichlet-lower}
 L_D(k)=\prod_p(1-\chi_D(p)p^{-k})^{-1}
 \geq\prod_p(1+p^{-k})^{-1}
 =\frac{\zeta(2k)}{\zeta(k)}\geq\frac1{\zeta(2)}.
\end{equation}
Dividing \eqref{eq:remainder-bound} by
$B_{k,e}|L_D(1-k)|$ and using
\eqref{eq:Dirichlet-FE-value}--\eqref{eq:Dirichlet-lower} gives exactly
\eqref{eq:effective-error}.

For every fixed positive $Q$, we have
\[
 \frac{Q^k}{\Gamma(k)}\longrightarrow 0\qquad(k\longrightarrow\infty).
\]
Hence \eqref{eq:endpoint-asymptotic} follows.
\end{proof}

\begin{remark}\label{rem:effective}
Every constant in \eqref{eq:effective-error} is explicit. 
\end{remark}

\section{A determinant limit calculation}
\label{sec:determinant}

Fix $D$ and $r$ as in the statement of Theorem \ref{thm:main}.  Let
$p_1<p_2<\cdots<p_r$ be $r$ primes not dividing $D$, that is, 
\begin{equation}\label{eq:primes}
 p_1<p_2<\cdots<p_r,\qquad p_j\nmid D.
\end{equation}
We call an integer $\ell$ {\em admissible} if $(-1)^\ell D > 0$.
For an admissible $\ell$ and $1\leq e\leq r$, put
$k_e:=\ell-2e$.  When $\ell\geq 2r+4$, all $k_e\geq 4$.  Since
$k_e\equiv\ell\pmod 2$, we have $(-1)^{k_e} D >0$.

For a prime $p\nmid D$, formula \eqref{eq:twisted-divisor-sum} gives
\begin{equation}\label{eq:prime-sigma}
 \sigma_{k-1,D,1}(p)=1+\chi_D(p)p^{k-1}.
\end{equation}
Applying Proposition~\ref{prop:endpoint} with $k=k_e$ and $n=p_j$, and
noting that $B_{k_e,e}=\binom{\ell-1}{2e}$, we obtain, uniformly for the
finitely many pairs $1\leq e,j\leq r$,
\begin{align}\label{eq:prime-asymptotic}
 \frac{A_{D,\ell-2e,e}(p_j)}
 {\binom{\ell-1}{2e}L_D(1-\ell+2e)}
 &=p_j^{2e}\bigl(1+\chi_D(p_j)p_j^{\ell-2e-1}\bigr)+o(1)\notag\\
 &=\chi_D(p_j)p_j^{\ell-1}+p_j^{2e}+o(1)
\end{align}
as $\ell\to\infty$ through the admissible parity class.

Define the $r\times r$ matrix
\begin{equation}\label{eq:Mell}
 M_\ell:=\left(
 \frac{A_{D,\ell-2e,e}(p_j)}
 {\binom{\ell-1}{2e}L_D(1-\ell+2e)}
 \right)_{1\leq e,j\leq r}.
\end{equation}
The denominators are nonzero by
\eqref{eq:Dirichlet-FE-value}.  We now compute the leading term of its
determinant.

\begin{lemma}\label{lem:determinant}
With $M_\ell$ as in \eqref{eq:Mell}, one has
\begin{equation}\label{eq:det-limit}
 \lim_{\substack{\ell\to\infty\\(-1)^\ell D>0}}
 p_r^{1-\ell}\det M_\ell=\chi_D(p_r)\Delta_r,
\end{equation}
where
\begin{align}\label{eq:Delta-matrix}
 \Delta_r:=
 \det\begin{pmatrix}
 p_1^2&p_2^2&\cdots&p_{r-1}^2&1\\
 p_1^4&p_2^4&\cdots&p_{r-1}^4&1\\
 \vdots&\vdots&&\vdots&\vdots\\
 p_1^{2r}&p_2^{2r}&\cdots&p_{r-1}^{2r}&1
 \end{pmatrix}\notag
\end{align}

\end{lemma}

\begin{proof}
Let
\[
 \one:=(1,\ldots,1)^T\in\C^r,
 \quad w_j:=(p_j^2,p_j^4,\ldots,p_j^{2r})^T,
 \quad v_{\ell,j}:=\chi_D(p_j)p_j^{\ell-1}.
\]
By \eqref{eq:prime-asymptotic}, the $j$th column of $M_\ell$ is
\begin{equation}\label{eq:column-form}
 v_{\ell,j}\one+w_j+\varepsilon_{\ell,j},
 \qquad \varepsilon_{\ell,j}\longrightarrow 0
 \quad(1\leq j\leq r).
\end{equation}
Put $z_{\ell,j}:=w_j+\varepsilon_{\ell,j}$.  Multilinearity of the
determinant gives
\begin{align}\label{eq:det-expansion}
 \det M_\ell
 &=\det(z_{\ell,1},\ldots,z_{\ell,r})\notag\\
 &\quad+\sum_{j=1}^rv_{\ell,j}
 \det(z_{\ell,1},\ldots,z_{\ell,j-1},\one,
 z_{\ell,j+1},\ldots,z_{\ell,r}).
\end{align}
There are no terms involving two or more $v_{\ell,j}$, because every such
determinant contains at least two columns equal to $\one$.

Divide \eqref{eq:det-expansion} by $p_r^{\ell-1}$.  The first determinant
remains bounded and hence tends to zero after this normalization.  For
$j<r$,
\[
 \frac{v_{\ell,j}}{p_r^{\ell-1}}
 =\chi_D(p_j)\left(\frac{p_j}{p_r}\right)^{\ell-1}\longrightarrow 0.
\]
The $j=r$ term converges to
\[
 \chi_D(p_r)\det(w_1,\ldots,w_{r-1},\one)
 =\chi_D(p_r)\Delta_r.
\]
This proves \eqref{eq:det-limit}.
\end{proof}

\begin{corollary}\label{cor:non-zero}
    We have \begin{equation}
 \lim_{\substack{\ell\to\infty\\(-1)^\ell D>0}}
 p_r^{1-\ell}\det M_\ell \ne 0. 
\end{equation}
\end{corollary}
\begin{proof}
First note that by the Vandermonde determinant 
    \[
\Delta_r =\left(\prod_{j=1}^{r-1}p_j^2\right)
 \left(\prod_{1\leq i<j\leq r-1}(p_j^2-p_i^2)\right)
 \left(\prod_{i=1}^{r-1}(1-p_i^2)\right).
    \]
This identity immediately implies that 
$\Delta_r\neq 0$.
Moreover, $p_r\nmid D$ implies $\chi_D(p_r)=\pm1$, so the limit in
\eqref{eq:det-limit} is nonzero.
\end{proof}

\begin{remark}
    Once the primes $p_1, \dots, p_r$ are chosen, we can effectively find a large enough $\ell$ for which $\det M_\ell \ne 0$. We can see this as follows. In the sequal all norms
below are Euclidean.  Put
\[
 \eta_\ell:=\max_{1\leq j\leq r}\|\varepsilon_{\ell,j}\|,
 \qquad W:=r p_r^{2r},\qquad Q_j:=8p_j^2|D|,
\]
and, with $M_j=p_j|D|$, define
\begin{align}\label{eq:rational-error}
 T_{e,j}(\ell)&:=\nu(D)C_eM_j^{2e+1}|D|
       \frac{Q_j^{\ell-2e}}{(\ell-2e-1)!},\notag\\
 h_\ell&:=r\max_{1\leq e,j\leq r}T_{e,j}(\ell).
\end{align}
The bound \eqref{eq:effective-error}, together with
$\zeta(2)/2<1$, $\sqrt{|D|}\leq |D|$, and $2\pi<8$, gives
$\eta_\ell\leq h_\ell$.  Also $\|w_j\|\leq W$ and
$\|\one\|=\sqrt r\leq r$.  Every quantity in
\eqref{eq:rational-error} is rational and can be evaluated exactly.
Moreover,
\begin{equation}\label{eq:error-ratio}
 \frac{T_{e,j}(\ell+2)}{T_{e,j}(\ell)}
 =\frac{Q_j^2}{(\ell-2e)(\ell-2e+1)}.
\end{equation}
Thus $h_\ell$ is decreasing along the admissible parity class as soon as
$\ell\geq 2r+Q_r$, and $h_\ell\to0$.

For columns $c_1,\ldots,c_r\in\C^r$, Hadamard's inequality states that
$|\det(c_1,\ldots,c_r)|\leq\prod_i\|c_i\|$.  Applying this inequality to
\eqref{eq:det-expansion} gives
\begin{equation}\label{eq:effective-det-bound}
\begin{aligned}
 \left|p_r^{1-\ell}\det M_\ell-\chi_D(p_r)\Delta_r\right|
 &\leq p_r^{1-\ell}(W+h_\ell)^r\\
 &\quad+r\sum_{j=1}^{r-1}
 \left(\frac{p_j}{p_r}\right)^{\ell-1}
 (W+h_\ell)^{r-1}\\
 &\quad+r(r-1)h_\ell(W+h_\ell)^{r-2}
 =:\mathcal R_\ell.
\end{aligned}
\end{equation}
The last summand is omitted when $r=1$.  To obtain it for $r\geq2$,
replace the columns $z_{\ell,1},\ldots,z_{\ell,r-1}$ in the $j=r$
determinant by $w_1,\ldots,w_{r-1}$ one at a time.  Each of the resulting
$r-1$ determinants has one error column of norm at most $h_\ell$, the
column $\one$ of norm at most $r$, and $r-2$ columns of norm at most
$W+h_\ell$.

Let $\ell_1$ be the least admissible integer at least $2r+Q_r$.
Then $\ell_1\geq2r+4$, and \eqref{eq:error-ratio} shows that
$\mathcal R_\ell$ is decreasing for admissible $\ell\geq\ell_1$.
It tends to zero, including when $r=1$.  Since $\mathcal R_\ell$ is explicit and $|\Delta_r|$ is a positive integer, we just need to find $\ell$ such that 
\[
 \mathcal R_\ell<|\Delta_r|/2,
 \qquad \ell=\ell_1,\ell_1+2,\ldots.
\]
Choose $\ell_0(D,r)$ to be the first
successful value.  Monotonicity guarantees the same inequality at every
subsequent admissible value, proving the effectivity assertion.
\end{remark}

\section{Proof of the main theorem}
\label{sec:spectral}

We now present the proof of Theorem~\ref{thm:main} and Corollary~\ref{cor:untwisted}. 

We first observe that for fixed $D$ and $r$, the forms $\mathcal F_{\ell,e}^{(D)}$, $1\leq e\leq r$,
are linearly independent for every sufficiently large admissible $\ell$. In fact, 
Corollary~\ref{cor:non-zero} shows that $\det M_\ell\neq 0$ for every
sufficiently large $\ell$ satisfying $(-1)^\ell D>0$.  The matrix of the
$p_1,\ldots,p_r$ Fourier coefficients of the forms
$\mathcal F_{\ell,e}^{(D)}$ is obtained from $M_\ell$ by multiplying its $e$th row
by the nonzero scalar
\[
 \binom{\ell-1}{2e}L_D(1-\ell+2e).
\]
It is therefore nonsingular.  A linear relation among the $r$ cusp forms
would give a linear relation among these $r$ coefficient rows, which is
impossible.

Let
$\mathcal B_{2\ell}$ be the normalized Hecke eigenbasis of
$S_{2\ell}(\SLtwoZ)$.  Distinct members of $\mathcal B_{2\ell}$ are
orthogonal for the Petersson inner product.

Fix $1\leq e\leq r$ and put $k=\ell-2e$.
Proposition~\ref{prop:Petersson-formula} gives
\begin{equation}\label{eq:spectral-factorization}
 \langle \mathcal F_{\ell,e}^{(D)},f\rangle
 =C_{D,\ell,e}L(f,2\ell-2e-1)L(f\otimes\chi_D,\ell),
\end{equation}
where $C_{D,\ell,e}\neq 0$.  Indeed, all gamma factors are nonzero,
$L_D(k)\neq 0$ by its Euler product, and $L_D(1-k)\neq 0$ by
\eqref{eq:Dirichlet-FE-value}.

In particular,
\begin{equation}\label{eq:spectral-vanishing}
 L(f\otimes\chi_D,\ell)=0
 \quad\Longrightarrow\quad
 \langle \mathcal F_{\ell,e}^{(D)},f\rangle=0.
\end{equation}

Let
\begin{equation}\label{eq:V-space}
 V_{\ell,D}:=\Span_{\C}\{f\in\mathcal B_{2\ell}:
 L(f\otimes\chi_D,\ell)\neq 0\}.
\end{equation}
Expanding $\mathcal F_{\ell,e}^{(D)}$ in the orthogonal basis
$\mathcal B_{2\ell}$ and using
\eqref{eq:spectral-vanishing}, we find that
\begin{equation}\label{eq:F-in-V}
 \mathcal F_{\ell,e}^{(D)}\in V_{\ell,D}\qquad(1\leq e\leq r).
\end{equation}
The linear independence of the forms $\mathcal F_{\ell,e}^{(D)}$, $1\leq e\leq r$ established above shows that 
$\dim_{\C}V_{\ell,D}\geq r$.  Because the normalized eigenforms are
linearly independent, $\dim V_{\ell,D}$ is exactly the number of
eigenforms with nonzero twisted central value.  This proves
\eqref{eq:nonvanishing-count}, and hence Theorem~\ref{thm:main}.

For $D=1$, equation \eqref{eq:GkD} becomes
\[
 G_{k,1}(z)=\frac{\zeta(1-k)}2+
 \sum_{n\geq 1}\sigma_{k-1}(n)q^n.
\]
Let $B_k$ denote the $k$th Bernoulli number.  If
\[
 E_k(z)=1-\frac{2k}{B_k}\sum_{n\geq 1}\sigma_{k-1}(n)q^n
\]
is the normalized level-one Eisenstein series, then
$\zeta(1-k)=-B_k/k$ gives
\begin{equation}\label{eq:G-E}
 G_{k,1}=\frac{\zeta(1-k)}2E_k.
\end{equation}
The scalar in \eqref{eq:G-E} is nonzero for even $k\geq 4$.  Since the
trace from level one to itself is the identity and the bracket is bilinear,
\begin{equation}\label{eq:F-bracket-E}
 \mathcal F_{\ell,e}^{(1)}=
 \left(\frac{\zeta(1-\ell+2e)}2\right)^2
 [E_{\ell-2e},E_{\ell-2e}]_{2e}.
\end{equation}
Taking $K=2\ell$ in Theorem~\ref{thm:main} proves
Corollary~\ref{cor:untwisted}.


\begin{thebibliography}{99}

\bibitem{Cohen}
H.~Cohen,
\emph{Sums involving the values at negative integers of $L$-functions of
quadratic characters},
Math. Ann. \textbf{217} (1975), 271--285.

\bibitem{Davenport}
H.~Davenport,
\emph{Multiplicative Number Theory}, 3rd ed., revised by H.~L. Montgomery,
Graduate Texts in Mathematics, vol.~74, Springer-Verlag, New York, 2000.

\bibitem{KayathEtAl}
J.~Kayath, C.~Lane, B.~Neifeld, T.~Ni, and H.~Xue,
\emph{Subspaces spanned by eigenforms with nonvanishing twisted central
$L$-values},
Canad. J. Math., published online 3 October 2025, 1--34,
\href{https://doi.org/10.4153/S0008414X25101697}
{doi:10.4153/S0008414X25101697}.


\bibitem{LanphierTaklooBighash}
D.~Lanphier and R.~Takloo-Bighash,
\emph{On Rankin--Cohen brackets of eigenforms},
J. Ramanujan Math. Soc. \textbf{19} (2004), no.~4, 253--259.

\bibitem{Rankin}
R.~A. Rankin,
\emph{The construction of automorphic forms from the derivatives of a
given form},
J. Indian Math. Soc. \textbf{20} (1956), 103--116.

\bibitem{ZagierDifferential}
D.~Zagier,
\emph{Modular forms and differential operators},
Proc. Indian Acad. Sci. Math. Sci. \textbf{104} (1994), 57--75.

\end{thebibliography}
\end{document}